\documentclass[preprint,12pt]{elsarticle}

\usepackage{amsthm,amssymb,amsmath,amsfonts,amsthm,amscd,graphicx} 
\usepackage{pgf,tikz,mathtools,relsize}
\usepackage{mathtools}
\usepackage{comment}
\usepackage{lipsum}  
\usepackage[normalem]{ulem}
\usepackage{mwe}
\usepackage{caption} 
\usepackage{subcaption}
\usepackage[export]{adjustbox}
\usepackage{multicol,color,bm}
\usepackage{lineno}
\usepackage{hyperref}

\newtheorem{Lemma}{Lemma}
\newtheorem{Theorem}{Theorem}

\newtheorem{Corollary}{Corollary}[section]

\newtheorem{Remark}{Remark}[section]

\newcommand{\R}{\mathbb{R}}

\renewcommand{\P}{{\rm P} }

\newcommand{\be}{\begin{equation}}
\newcommand{\ee}{\end{equation}}
\newcommand{\bea}{\begin{eqnarray}}
\newcommand{\eea}{\end{eqnarray}}
\newcommand{\beas}{\begin{eqnarray*}}
\newcommand{\eeas}{\end{eqnarray*}}
\newcommand{\ba}{\begin{array}}
\newcommand{\ea}{\end{array}}

\newcommand{\lip}[2]{\left({}#1,#2\right){}}
\newcommand{\norm}[1]{\ensuremath{\left\|{#1}\right\|}}
\newcommand{\trinorm}[1]{{\left\vert\kern-0.15ex\left\vert\kern-0.15ex\left\vert #1 
    \right\vert\kern-0.15ex\right\vert\kern-0.15ex\right\vert}}
\newcommand{\pare}[1]{\left({}#1\right)}
\newcommand{\curly}[1]{\left\{{}#1\right\}}
\newcommand{\abs}[1]{\ensuremath{\left\lvert{#1}\right\rvert}}

\newcommand{\inv}{^{-1}}
\newcommand{\chalf}{^{n+\frac{1}{2}}}

\newcommand{\divergence}{\nabla \cdot}
\newcommand{\Grad}{\ensuremath{\nabla}}

\newcommand{\Rey}{\ensuremath{\mathrm{Re}}}
\newcommand{\ohm}{ \Omega }

\newcommand{\bfe}{\ensuremath{\bm {e}}}
\newcommand{\bff}{\ensuremath{\mathbf{f}}}

\newcommand{\bfx}{\ensuremath{\mathbf{x}}}

\newcommand{\bfu}{\ensuremath{\bm {u}}}
\newcommand{\bfv}{\ensuremath{\bm {v}}}
\newcommand{\bfw}{\ensuremath{\bm {w}}}

\newcommand{\bfW}{\ensuremath{\bm {W}}}

\newcommand{\bfX}{\ensuremath{\bm {X}}}
\newcommand{\bfV}{\ensuremath{\bm {V}}}

\newcommand{\delt}{\ensuremath{\Delta t}}

\newcommand{\bfphi}{\ensuremath{\boldsymbol{\phi}}}
\newcommand{\bfeta}{\ensuremath{\boldsymbol{\eta}}}

\newcommand{\bw}{\bm w}

\newcommand{\bv}{\bm v}

\journal{ ArXiv}

\begin{document}

\begin{frontmatter}



\title{Longer time accuracy for the Ladyzhenskya model with the EMAC formulation}


\author[OUC]{Rihui Lan}
\ead{lanrihui@ouc.edu.cn}

\author[VT]{Jorge Reyes\corref{cor1}}
\ead{reyesj@vt.edu}


\cortext[cor1]{Corresponding author}


\address[OUC]{School of Mathematical Sciences, Ocean University of China, Qingdao 266100, Shandong, China}

\address[VT]{Department of Mathematics, Virginia Tech, Blacksburg,VA 24061, USA}

\begin{abstract}

In this paper, we incorporate the EMAC formulation into the Ladyzhenskaya model (LM), a large eddy simulation (LES) of incompressible flows. The EMAC formulation, which conserves energy, linear momentum, and angular momentum even with weak enforcement of incompressibility, has been shown to provide tangible benefits over the popular skew-symmetric for direct numerical simulation and regularized models of the Navier Stokes equations (NSE).
The combination of EMAC with the LM addresses the known over-dissipation issues associated with the classical Smagorinsky model (SM). We develop a finite element discretization for the EMAC-LM system and analyze its stability and derive numerical error estimates, showing improved long-time behavior compared to the standard LM approach, particularly due to EMAC’s favorable Gronwall constant independent of the Reynolds number. Benchmark simulations demonstrate that the EMAC-LM model yields more accurate flow structures, especially at high Reynolds numbers.

\end{abstract}


\end{frontmatter}

\section{Introduction} \label{S:1}

Turbulent flows are characterized by motion at many different length scales, from large swirling structures down to tiny eddies. Capturing all of these scales requires solving the Navier–Stokes equations (NSE) on an extremely dense computational grid, thus direct numerical simulation (DNS) quickly becomes impractical outside of very limited cases. Large eddy simulation (LES) are one possible way to addresses this challenge by focusing only on the dominant, large-scale motions and approximating the influence of the smaller ones. This approach rests on the observation that larger eddies evolve in a more orderly and deterministic fashion, while smaller eddies behave in an irregular, chaotic way \cite{berselli2006mathematics}. Even so, the smaller eddies are essential because they remove energy from the large-scale structures, so their contribution must still be modeled.

In this work, we use an LES model that incorporates an artificial viscosity term to account for the dissipative influence of unresolved small-scale fluctuations. The model we study is Ladyzhenskya Model, which can be viewed as a generalization of the classical Smagorinky model. This distinction matters because, while the Smagorinsky model is widely used in industry, it is known to be overly dissipative \citep{ervin2007numerical,iliescu2003numerical,zang1993dynamic}, motivating the development and analysis of improved variants.

The Ladyzhenskaya model and the Smagorinsky model have been the subject of extensive theoretical and numerical study \citep{sagaut2006large,iliescu2002convergence,john2003large}. The well-posedness of these models was first explored by Ladyzhenskaya \citep{ladyzhenskaya1967new,ladyzhenskaya1968modifications}, and further developed by Du and Gunzburger in \citep{du1990finite,du1991analysis}. More recently, a ``corrected" Smagorinsky variant was proposed in \cite{siddiqua2022numerical}, aiming to better capture energy transfer from unresolved to resolved scales while preserving well-posedness. A particular case of the Ladyzhenskaya model, called the generalized Smagorinsky model, was studied in \cite{ReyesGSM}. In \cite{cao2022continuous}, 3D continuous data assimilation was studied using the Ladyzhenskaya model, while recently a reduced order Ladyzhenskaya model was introduced in \cite{reyes2025verifiability}.

The Ladyzhenskya model (LM) equations on domain $\Omega \subset \R^d\; (d=2,3)$ with boundary $\Gamma$ and time interval $[0,T]$ are
\begin{equation}\label{LM}
\begin{aligned}
    \bfw_t - \nabla \cdot ( \Rey^{-1} + { (C_S \delta)^r \| \nabla\bfw\|_F^s})\nabla\bfw  + \nabla p +  (\bfw \cdot \nabla)\bfw  = \bff,  \textrm{ in } (0,T] \times \Omega,  \\
    \nabla\cdot \bfw = 0, \textrm{ in } [0,T] \times \Omega, 
    \end{aligned}
\end{equation}
accompanied by appropriate initial condition and boundary condition. Herein, $\bfw(\bfx,t)$ and $p(\bfx,t)$ denote the fluid's velocity and pressure at spatial coordinate $ \bfx \in \ohm$ and time $ t\in [0,T]$. $\Rey$ is the Reynolds number, $\bff$ is the forcing term, and $C_S$ is the Smagorinsky constant. The spatial filter radius is denoted by $\delta$, and the Frobenius norm by $ \norm{ \mathbf{A} }_F =  \pare{ \sum_{i,j = 1}^d a_{ij}^2 }^{\tfrac{1}{2}} $.
It is clear that when $r=2$ and $s=1$ return to the classical Smagorinsky model. 
As for the well-posedness of \eqref{LM}, Ladyzhenskaya \citep{ladyzhenskaya1967new,ladyzhenskaya1968modifications} first showed that the model has a globally unique solution in time for any Reynolds number and any $ s \geq \frac{1}{2}$.
This was then improved to $ s \geq \frac15 $ in \cite{du1991analysis} where they also showed that weak solutions to the stationary case are unique under certain conditions on the Reynolds number.
The accuracy of \eqref{LM} approximating to NSE was studied by Du and Gunzburger \cite{du1991analysis}.
They showed in Theorem 4.3 that as $\nu_1 = (C_S\delta)^r \to 0$ 
the solution of the LM model weakly converges to the NSE solution. This observation is particularly relevant in our setting, since the constant $(C_S \delta)^r \to 0$ as $\delta \to 0$, with $\delta = O(h)$, where $h$ denotes the mesh size of the finite element triangulation.

Recalling that the NSE are derived from conservation of mass and linear momentum, they are also known to conserve several other physical quantities such as energy, angular momentum, enstrophy (in 2D), and helicity (in 3D) \cite{GS98}. These conserved quantities are believed to play a fundamental role in the evolution of flow structures, the energy cascade, and dissipation at the microscale \cite{GS98,Rebholz07}. However, standard Galerkin discretizations often fail to preserve these properties, due to how the nonlinearity is handled \cite{CHOR17,CHOR19}.Various nonlinear formulations—convective, skew-symmetric, rotational, and conservative—have been proposed to address this issue, with the skew-symmetric and rotational forms being commonly used to improve energy behavior. One solutions is to use strongly divergence-free discretizations like Scott–Vogelius elements, though they often require mesh restrictions and high-degree polynomials, especially on quadrilateral meshes. More recently, Charnyi et al.\ introduced the EMAC formulation (short for Energy, Momentum, and Angular Momentum Conservation), which provides a new structure for the nonlinear term seen in \eqref{eq: emacform} and better preserves these invariants across a variety of time-stepping schemes \cite{CHOR17}. 
\begin{equation}\label{eq: emacform}
    \bfu \cdot \Grad \bfu + \Grad p = 2 D(\bfu)\bfu + (\divergence \bfu)\bfu + \Grad P,
\end{equation}
where $ P = p - \frac{1}{2}|\bfu|^2$, and $ D(u) = \frac{\Grad \bfu + (\Grad \bfu)^T}{2}$.
Studies have shown that the EMAC formulations offer better long-time behavior compared to traditional skew-symmetric formulations \cite{belding2022efficient,olshanskii2020longer,LeoEMAC,vonwahl2025benchmarkstresstestsflow,lan2025robust}. Motivated by these results in high Reynolds number simulations, we apply the EMAC formulation to the Ladyzhenskaya LES model and compare its performance with the skew-symmetric scheme. Our goal is to assess whether EMAC’s conservation properties can improve stability and accuracy for LES models on coarser meshes.

The rest of this article will be organized as follows. Section \ref{Prelim_Not} we give necessary background on notation and preliminaries.
Section \ref{schemeStab} gives the finite element scheme and its unconditional stability. In Section
\ref{erroranalysis} the finite element convergence error analysis is presented. Then, in Section \ref{sec-numres}, we present numerical experiments that support the theoretical results. These include a study of convergence rates and of two-dimensional flow past a step. We conclude our work in the last Section \ref{sec-Con}.

\section{Notations and Preliminaries} \label{Prelim_Not}

This paper will use the following spaces: $L^p(\Omega) $, $W^{k,p}(\Omega) $, and $ H^k(\Omega) = W^{k,2}(\Omega)$, where $ k \in \mathbb{N}, 1 \leq p < \infty $ for domain $\Omega \subset\mathbb{R}^d$ $(d=2,3)$. The $L^2(\Omega)$ norm is denoted as $\Vert \cdot \Vert$ with the corresponding inner product $(\cdot,\cdot)$. The $L^p(\Omega)$ norm is denoted by $\| \cdot \|_{p} $ while the Sobolev $ W^{k,p}(\Omega)$ norm is $ \Vert \cdot \Vert_{k,p} $. The solutions are sought in the following functional spaces:
\begin{eqnarray*}
    & &\mathrm{Velocity\;space} -\bfX_s:= W_0^{1,s+2}(\Omega) = \curly{ \bfu \in W^{1,s+2}(\ohm): \bfu \mid_{\partial\ohm} =0 \;\text{in}\; L^2(\ohm)}\, , \\
    & &\mathrm{Pressure\;space} - Q:=L^{2}_{0}(\Omega)=\left\{q \in L^2(\Omega) :
     \int_\Omega \,  q \, d\Omega = 0 \right\} \, , \\
    & &\mathrm{Divergence-free\;space} - \bfV_s:=\left\{\bfv \in \bfX_s: \int_\Omega q \, \Grad\cdot
    \bfv  \, d\Omega \, = \, 0, \; \forall  q \in Q \right\},
\end{eqnarray*}
where $s$ is one of the model's parameters as defined above in \eqref{LM}.
$X_s'$ is the dual space of $X_s$ and the norm of the space is $\| \cdot \|_{-1} $. 
Moreover we set $ \bfX = W^{1,2}_0 (\ohm)$ and $ \bfV$ to be the weakly divergence-free subspace of $\bfX$.

Let $\Omega $ be a polygonal domain and following the work from \cite{ReyesGSM} we consider $ \bfX_h \subset \bfX_s$ and $ Q_h \subset Q$, with
\beas
    \bfV_{h} &:=& \left\{ \bfu \in \bfX_{h} : (q , \Grad \cdot \bfu) = 0 ,
    \; \forall q \in Q_{h} \right\}.
\eeas
Discrete inf-sup condition \cite{gunzburger1989finite} gives that $\exists \; \gamma \in \R^+$, such that
\begin{equation}\label{dinfsup}
   \gamma  \le \inf_{q_{h} \in Q_{h}} \sup_{\bfv_h \in \bfX_{h}}
   \frac{ (q_{h} \, , \Grad \cdot \bfv_h ) }{\|q_{h}\| \, \| \nabla \bfv_h \|} \ .
\end{equation}
Taylor-Hood and mini element are examples of such spaces \citep{BrennerScott,gunzburger1989finite}.

Additionally we use the following notation  $ \bfu (t\chalf) = \bfu \left( \dfrac{t^{n+1} + t^n}{2} \right) $, and $ \bfu\chalf = \frac{1}{2}\pare{\bfu^{n+1} + \bfu^n}$ for both the discrete and continuous variables.
Furthermore, let $\delt$ denote the time step, then $t^{n}
= n \delt$, $n = 0, 1, \dots, M$, and final time is $T := M \delt$. Furthermore,
\begin{eqnarray*}
 \trinorm{ \bfu }_{\infty, k} &:=& \max_{0 \le n \le M} \|\bfu^{n}\|_{k,2} \, ,
    \hspace{0.5in}  \trinorm{ \bfu_{1/2} }_{\infty, k} \ := \ \max_{1 \le n \le M} \|\bfu\chalf\|_{k,2} \, ,\\
    \trinorm{ \bfu }_{m,k} &:=& \left(\sum_{n=0}^{M} \| \bfu^{n} \|^{m}_{k,2}\delt\right)^{1/m}, \, 
    \trinorm{ \bfu_{1/2} }_{m,k} \ := \ \left(\sum_{n=1}^{M} \| \bfu\chalf \|^{m}_{k,2}\delt\right)^{1/m}\, .
\end{eqnarray*}
For $ \bfu,\bfv,\bfw \in \bfX_s $, define the trilinear forms $ b, b^*,c: \bfX_s \times \bfX_s \times \bfX_s \mapsto \mathbb{R}  $ by
\begin{eqnarray*}
 b(\bfu,\bfv,\bfw) &=& (\bfu\cdot\nabla \bfv,\bfw),\\
 b^{*}(\bfu,\bfv,\bfw)&=&  \frac{1}{2}(b(\bfu,\bfv,\bfw)-b(\bfu,\bfw,\bfv)),\\
 c(\bfu,\bfv,\bfw) &=& 2\pare{D(\bfu)\bfv,\bfw} + \pare{(\divergence\bfu)\bfv,\bfw},
 \end{eqnarray*}
where $ D(\bfu)$ is defined as the symmetric part of $ \Grad \bfu$ by $ D(\bfu) = \frac{\Grad \bfu + (\Grad \bfu)^T}{2}$. These trilinear forms will be used to represent the non-linear term in the LM model for both the SKEW and EMAC scheme.

We assume the following approximation properties, \citep{BrennerScott}:
\beas
    \inf_{\bfv\in \bfX_{h}}\| \bfu - \bfv \| &\le& C h^{k+1} \| \bfu \|_{{k+1},2},\;\; \bfu \in
    H^{k+1}(\Omega)^{d}, \nonumber \\
    \inf_{\bfv\in \bfX_{h}}\|\Grad( \bfu - \bfv) \| &\le& C h^{k} \| \bfu \|_{{k+1},2},\;\; \bfu \in
    H^{k+1}(\Omega)^{d}, \nonumber \\
    \inf_{q \in Q_{h}} \| p - q \| &\le& Ch^{s+1} \| p \|_{s+1,2},\;\;  p \in
    H^{s+1}(\Omega). \nonumber \label{PROP}
\eeas

The Stokes projection ${\rm P}_{st}:\bm V\rightarrow \bm V_h$ \cite{girault2003quasi,GiraultNochettoScott}  is defined by
\begin{equation}\label{stpr}
\left(\nabla {\rm P}_{st}\bw,  \nabla\bv_h\right) = \left( \nabla\bw, \nabla\bv_h\right), \quad \forall\, \bw\in \bm V,\;\bv_h \in \bm V_h.
\end{equation}
Regarding the Stokes projection \eqref{stpr}, a stability lemma was proved in \cite{GiraultNochettoScott}, and an error estimate was shown in \cite{de2018analysis} in conjunction with the Aubin-Nitsche lift technique \cite{BakerDougalisKarakashian}. 
\begin{Lemma} [Stability of Stokes Projection \cite{GiraultNochettoScott}]
	\label{stokes_stability}
	For any  fixed $r \in [2,\infty]$,
	there exists a constant $C >0$ depending only on $\Omega$  such that
	\begin{equation*}\label{333:stab}
	\Vert \nabla {\rm P}_{st} \bw\Vert_{r} \leq C \Vert\nabla \bw\Vert_{r}, \quad \forall\,\bw \in \bm V.
	\end{equation*}
\end{Lemma}
\begin{Lemma}[Error estimate of Stokes Projection \cite{de2018analysis} ]
	\label{stokes_error_est}
	There exists a constant $C >0$ depending only on $\Omega$  such that the following estimate holds:
	\begin{equation*}\label{444}
	\Vert\bw - {\rm P}_{st} \bw\Vert + h \Vert\bw - {\rm P}_{st} \bw\Vert_{1,2} \leq C h^3 \Vert \bw \Vert_{3,2},  \quad\forall\,\bw\in \bm V \cap (H^3(\Omega))^d.
	\end{equation*}
\end{Lemma}
Let us present some popular inverse inequalities as follows \cite{BrennerScott}:
\begin{Lemma}[Inverse inequalities \cite{BrennerScott}]\label{lem: Inverseineq}
	For any fixed $0 \leq n \leq m \leq 1, 1 \leq q \leq p \leq \infty$, there exists a constant $C>0$, depending only on $\Omega$,  such that 
	\begin{equation*}\label{555:inv}
	\left\|\bw_h\right\|_{m, p} \leq C  h^{n-m-d\left(\frac{1}{q}-\frac{1}{p}\right)}\left\|\bw_h\right\|_{n, q}, \quad \forall\, \bw_h \in \bm X_h.
	\end{equation*}
\end{Lemma}

We also use the following lemmas.

\begin{Lemma}\label{C_TRIL} \citep{REYESExamples,LeoEMAC}
For $\bfu, \bfv, \bfw \in \bfX_s$ the trilinear term $c(\bfu, \bfv, \bfw)$ can be bounded by
\beas
    c(\bfu,\bfv,\bfw) &\leq&  C(\ohm) \norm{\Grad \bfu }\norm{\Grad \bfv} \norm{\Grad \bfw}, \label{gradbound}\\
    c(\bfu,\bfv,\bfw) &\leq& C(\ohm) \| \nabla \bfu \| \| \nabla \bfv \| \| \bfw \|^{1/2} \| \nabla \bfw \|^{1/2}, \label{cbound}\\
    c(\bfu,\bfv,\bfw) &\leq& C(\ohm) \| \Grad \bfu \| \|  \bfv \|^{1/2}  \| \nabla \bfv \|^{1/2} \| \nabla \bfw \|  . \label{cbound2}
\eeas
\end{Lemma}
Also from \cite{REYESExamples} we will utilize the following identity 
\begin{equation}
    c(\bfu,\bfu,\bfv) = c(\bfu,\bfv,\bfu) = -\frac{1}{2} c(\bfv,\bfu,\bfu). \label{eq: emac_id}
\end{equation}

\begin{Lemma} The trilinear term $c(\bfu, \bfv, \bfw)$ can be bounded as:\label{C infty bound}

\noindent For $ \bfv,\bfw\in \bfX_s$ and $ \Grad \bfu \in L^\infty(\ohm) $
\begin{equation}\label{eq: C infbound 1}
    c\pare{\bfu,\bfv,\bfw} \leq C(\ohm) \norm{\Grad \bfu}_{\infty} \norm{\bfv} \norm{\bfw}
\end{equation}
For $ \bfu,\bfw\in \bfX_s$ and $ \bfv \in L^\infty(\ohm) $
\begin{equation}\label{eq: C infbound 2}
    c(\bfu,\bfv,\bfw) \leq  C(\ohm) \norm{\Grad \bfu} \norm{\bfv}_{\infty} \norm{\bfw}
\end{equation}
For $ \bfu,\bfv\in \bfX_s$ and $ \bfw \in L^\infty(\ohm) $
\begin{equation}\label{eq: C infbound 3}
    c(\bfu,\bfv,\bfw) \leq C(\ohm) \norm{\Grad \bfu} \norm{\bfv} \norm{\bfw}_{\infty}
\end{equation}

\begin{proof}
The proof relies on an application of H\"olders inequality followed by using the appropriate Sobolev inequalities
    \beas
         c(\bfu,\bfv,\bfw) &=& 2\pare{D(\bfu)\bfv,\bfw} + \pare{(\divergence\bfu)\bfv,\bfw} \\
         &\leq & C\norm{D(\bfu)}_{\infty} \norm{\bfv} \norm{\bfw} +  C\norm{\divergence\bfu}_{\infty} \norm{\bfv} \norm{\bfw}\\
          &\leq & C\norm{\Grad \bfu}_{\infty} \norm{\bfv} \norm{\bfw}.
    \eeas
Which proves \eqref{eq: C infbound 1}, the proofs of \eqref{eq: C infbound 2} and \eqref{eq: C infbound 3} follow similar arguments.
\end{proof}
\end{Lemma}

\begin{Lemma}[Strong monotonicity \citep{minty1962monotone,lions1969quelques}]
\label{lemma1}
For $ \bfu, \bfv, \bfw \in W^{1,s+2}(\Omega)$, there exists a constant depending on $d,s$ and $ \Omega$, such that the following inequality holds
\begin{gather}
\label{StrongMonotonicty}
    \left( \| \Grad \bfu  \|_F^{s} \Grad \bfu - \| \Grad \bfv  \|_F^{s} \Grad \bfv , \Grad(\bfu - \bfv)  \right) \geq C \| \Grad (\bfu - \bfv)  \|_{s+2}^{s+2}.
\end{gather}
\end{Lemma}

\begin{Lemma}[ Lipschitz Continuity \cite{du1991analysis,ReyesGSM}]
\label{dufunzburgerlemma}
There exist constant $ M >0 $, and $ M_s >0 $ such that
\begin{gather}
\label{splus2bound}
 | (\|\Grad \bfu \|_F^{s} \Grad \bfu, \Grad\bfw ) - (\|\Grad \bfv \|_F^{s} \Grad \bfv, \Grad\bfw ) |  \nonumber\\
 \leq M_s( \| \bfu \|_{1,s+2} + \| \bfv \|_{1,s+2}  )^s \| \bfu - \bfv \|_{1,s+2} \| \bfw \|_{1,s+2} , \: \forall \bfu,\bfv,\bfw \in W^{1,s+2},
\end{gather}
and,
\begin{gather}
 | (\|\Grad \bfu \|_F^{s} \Grad \bfu, \Grad\bfw ) - (\|\Grad \bfv \|_F^{s} \Grad \bfv, \Grad\bfw ) |  \nonumber\\\label{infbound2}
 \leq M  ( \| \bfu \|_{1,\infty} + \| \bfv \|_{1,\infty}  )^s \| \bfu - \bfv \|_{1,2} \| \bfw \|_{1,2}, \: \forall \bfu,\bfv \in W^{1,\infty}, \bfw \in W^{1,s+2}.
\end{gather}

\end{Lemma}

The following lemmas are used in the analysis as well, \cite{heywood1990finite}.

\begin{Lemma}\label{gradErr}
Assume $\bfu,\Grad\bfu$ and $p$ $\in C^{0}(t^{n},t^{n+1};L^{2}(\Omega ))$.\\
If $u$ is twice
differentiable in time and $\bfu_{tt}\in L^{2}((t^{n},t^{n+1})\times \Omega )$,
then \label{lmubtmunm}
\begin{equation}
\norm{ \bfu\chalf-\bfu(t\chalf) }^{2}\ \leq \ \frac{1}{48}(\Delta
t)^{3}\,\int_{t^{n}}^{t^{n+1}}\Vert \bfu_{tt}\Vert ^{2}\,dt\,.  \label{apxnm0}
\end{equation}%
If $\bfu_{t}\in C^{0}(t^{n},t^{n+1};L^{2}(\Omega ))$ and $\bfu_{ttt}\in
L^{2}((t^{n},t^{n+1})\times \Omega )$, then
\begin{equation}
\norm{ \frac{\bfu^{n+1}-\bfu^{n}}{\Delta t}-\bfu_{t}(t\chalf) }^{2}\ \leq \
\frac{1}{1280}(\Delta t)^{3}\,\int_{t^{n}}^{t^{n+1}}\Vert \bfu_{ttt}\Vert
^{2}\,dt,\,\mbox{ \text{and} }  \label{apxnm05}
\end{equation}%
if $\nabla \bfu\in C^{0}(t^{n},t^{n+1};L^{2}(\Omega ))$ and $\nabla \bfu_{tt}\in
L^{2}((t^{n},t^{n+1})\times \Omega )$, then
\begin{equation}
\norm{ \nabla (\bfu\chalf-\bfu(t\chalf)) }^{2}\ \leq \ \frac{(\Delta t)^{3}%
}{48}\int_{t^{n}}^{t^{n+1}}\Vert \nabla \bfu_{tt}\Vert ^{2}\,dt\,.
\label{apeq1}
\end{equation}
if $ p \in L^{2}((t^{n},t^{n+1})\times \Omega )$, then
\begin{equation}
\norm{ (p\chalf-p(t\chalf)) }^{2}\ \leq \ \frac{(\Delta t)^{3}%
}{48}\int_{t^{n}}^{t^{n+1}}\Vert p_{tt}\Vert ^{2}\,dt\,.
\label{aprox p}
\end{equation}
\end{Lemma}

The proof lemma \ref{gradErr} is based on an application of an appropriate Taylor expansion with integral remainder.

\begin{Lemma}[Discrete Gronwall Lemma \cite{heywood1990finite}]
\label{discreteGronwall} Let $\Delta t$, H, and $a_{n},b_{n},c_{n},d_{n}$
(for integers $n \ge 0$) be finite nonnegative numbers such that
\begin{equation}
a_{l}+\Delta t \sum_{n=0}^{l} b_{n} \le \Delta t \sum_{n=0}^{l} d_{n}a_{n} +
\Delta t\sum_{n=0}^{l}c_{n} + H \ \ for \ \ l\ge 0. \label{gronwall1}
\end{equation}

Suppose that $\Delta t d_n < 1 \; \forall n$. Then,
\begin{equation}
a_{l}+ \Delta t\sum_{n=0}^{l}b_{n} \le \exp\left( \Delta t\sum_{n=0}^{l} \frac{d_{n}}{1 - \Delta t d_n } \right) \left( \Delta t\sum_{n=0}^{l}c_{n} + H
\right)\ \ for \ \ l \ge 0.
\end{equation}
\end{Lemma}

\section{Numerical Scheme and Analysis} \label{schemeStab}

The fully discrete SKEW formulation based on Crank-Nicolson time discretization is: Find $\bfw^{n+1}_h \in X_h$ and $p^{n+1}_h \in Q_h $ such that

\begin{eqnarray} \label{FDLM}
    \begin{aligned}
   \left(\frac{\bfw_h^{n+1} - \bfw_h^n }{\Delta t }, \bfv_h \right) + \Rey^{-1}(\nabla \bfw\chalf_h, \nabla \bfv_h)+
   (C_S \delta)^r ( \| \nabla \bfw\chalf_h  \|_F^s  \nabla \bfw\chalf_h, \nabla \bfv_h)  \\ + b^*(\bfw\chalf_h,\bfw\chalf_h,\bfv_h) -
   ( p^{n+1}_h, \nabla \cdot \bfv_h) = (\bff(t\chalf),\bfv_h), \quad \forall \bfv_h \in X_h, \\
   (\nabla \cdot \bfw_h^{n+1}, q^h)=0, \quad \forall q_h  \in Q_h .
    \end{aligned}
    \:
\end{eqnarray}

The fully discrete EMAC formulation based on Crank-Nicolson time discretization is: Find $\bfw^{n+1}_h \in X_h$ and $P^{n+1}_h \in Q_h $ such that

\begin{eqnarray} \label{FDLM_EMAC}
    \begin{aligned}
   \left(\frac{\bfw_h^{n+1} - \bfw_h^n }{\Delta t }, \bfv_h \right) + \Rey^{-1}(\nabla \bfw\chalf_h, \nabla \bfv_h)+
   (C_S \delta)^r ( \| \nabla \bfw\chalf_h  \|_F^s  \nabla \bfw\chalf_h, \nabla \bfv_h)  \\ + c(\bfw\chalf_h,\bfw\chalf_h,\bfv_h) -
   ( P^{n+1}_h, \nabla \cdot \bfv_h) = (\bff(t\chalf),\bfv_h), \quad \forall \bfv_h \in X_h, \\
   (\nabla \cdot \bfw_h^{n+1}, q^h)=0, \quad \forall q_h  \in Q_h .
    \end{aligned}
    \:
\end{eqnarray}

Using the discrete inf-sup condition \eqref{dinfsup}, we can consider the equivalent problem: Find $\bfw_h^{n+1} \in V_h$ such that
\begin{eqnarray} \label{DivfreeFDLM}
\begin{aligned}
   \left(\frac{\bfw_h^{n+1} - \bfw_h^n }{\Delta t }, \bfv_h \right) + \Rey^{-1}(\nabla \bfw\chalf_h, \nabla \bfv_h)+
   (C_S \delta)^r ( \| \nabla \bfw\chalf_h  \|_F^s  \nabla \bfw\chalf_h, \nabla \bfv_h)  \\ + b^*(\bfw\chalf_h,\bfw\chalf_h,\bfv_h)  = (\bff(t\chalf),\bfv_h), \quad \forall \bfv_h \in V_h.
\end{aligned} \quad
\end{eqnarray}
\begin{eqnarray} \label{DivfreeFDLM_EMAC}
\begin{aligned}
   \left(\frac{\bfw_h^{n+1} - \bfw_h^n }{\Delta t }, \bfv_h \right) + \Rey^{-1}(\nabla \bfw\chalf_h, \nabla \bfv_h)+
   (C_S \delta)^r ( \| \nabla \bfw\chalf_h  \|_F^s  \nabla \bfw\chalf_h, \nabla \bfv_h)  \\ + c(\bfw\chalf_h,\bfw\chalf_h,\bfv_h)  = (\bff(t\chalf),\bfv_h), \quad \forall \bfv_h \in V_h.
\end{aligned} \quad
\end{eqnarray}

\begin{Lemma}\label{stability:lem}
Both the SKEW and EMAC formulations of the Ladyzhenskaya Model are unconditionally stable: for any $\Delta t>0$, solutions of \eqref{DivfreeFDLM} and \eqref{DivfreeFDLM_EMAC} satisfy

\begin{equation}\label{Stability}
    \begin{gathered}
        \|\bfw^{M}_h \|^2 + \Rey\inv \delt \sum_{n=0}^{M-1}  \norm{ \Grad   \bfw\chalf_h }^2  + C (C_S \delta)^r \sum_{n=0}^{M-1} \norm{\Grad   \bfw\chalf_h}_{s+2}^{s+2}  \\
        \leq \|\bfw_0 \|^2 + \frac{\Delta t \Rey}{2} \sum_{n=0}^{M-1} \|\bff\chalf\|_{-1}^2 .
    \end{gathered}
\end{equation}

 Furthermore, for any $ \delt > 0 $ solutions exist, and as long as $ \delt < O( h^{1+\frac{d}{2}}) $ solutions are guaranteed to be unique. 

\end{Lemma}

\begin{proof}
Choosing $ \bfv_h = \bfw_h\chalf $ into both \eqref{DivfreeFDLM} and \eqref{DivfreeFDLM_EMAC} yields
\beas
 \frac{1}{2 \Delta t}
\left( \|\bfw_h^{n+1} \|^2 - \| \bfw_h^n \|^2 \right)  + \Rey^{-1} \norm{ \nabla \bfw\chalf_h }^2\\
+ (C_S \delta)^r \pare{ \norm{ \Grad\bfw\chalf_h }_F^s  \nabla\bfw\chalf_h , \nabla\bfw\chalf_h } 
  = (\bff(t\chalf),\bfw\chalf_h).
\eeas
This is because $ b^*(\bfw\chalf_h,\bfw\chalf_h, \bfw\chalf_h) = 0 $ and $ c(\bfw\chalf_h,\bfw\chalf_h, \bfw\chalf_h) = 0 $  by definition. Lemma \ref{lemma1} gives
\begin{equation}
    \lip{\norm{\nabla \bfw\chalf_h }_F^s \Grad\bfw\chalf_h }{ \nabla\bfw\chalf_h }
\geq C \norm{ \Grad\bfw\chalf_h }^{s+2}_{s+2}.
\end{equation}
Hence,
\beas
\frac{1}{2 \Delta t}\left( \|\bfw_h^{n+1} \|^2 - \| \bfw_h^n \|^2 \right)  &+& \Rey^{-1} \norm{ \Grad  \bfw\chalf_h }^{2} + C (C_S \delta)^{r}  \norm{ \Grad\bfw\chalf_h }^{s+2}_{s+2 }  \\
 &=& \pare{\bff(t\chalf),\bfw\chalf_h } \\
 &\leq& \norm{ \bff\chalf }_{-1}\norm{\Grad   \bfw\chalf_h} \\
 &\leq& \frac{\Rey^{-1}}{2}  \norm{ \Grad   \bfw\chalf_h }^2 + \frac{\Rey}{2} \norm{ \bff\chalf }_{-1}^2 . \\ 
\eeas
Rearranging some terms yields
\begin{multline*}
\frac{ \|\bfw_h^{n+1} \|^2 - \| \bfw_h^n \|^2 }{2\delt} +\frac{\Rey^{-1}}{2}  \norm{ \Grad   \bfw\chalf_h }^2 
+ C (C_S \delta)^r \norm{ \Grad \bfw\chalf_h }_{s+2}^{s+2} 
\leq \frac{\Rey}{2} \| \bff\chalf  \|_{-1}^2 .
\end{multline*}
Multiplying by $2 \Delta t$ and summing from $0$ to $M-1$, yield after simplification,
\beas
\|\bfw_h^{M} \|^2 
+ \Delta t\Rey^{-1} \sum_{n=0}^{M-1}  \norm{ \Grad   \bfw\chalf_h }
+ C (C_S \delta)^r \sum_{n=0}^{M-1} \norm{  \Grad\bfw\chalf_h }_{s+2}^{s+2}
\\
\leq  \| \bfw_h^0 \|^2  + \Delta t \Rey \sum_{n=0}^{M-1} \|\bff\chalf\|_{-1}^2 .
\eeas
With this unconditional stability, Leray-Schauder can be used to infer solutions to the LM at any time step, in an analogous way to what is done in \cite{layton2008introduction}.  For uniqueness of LM solutions, suppose there are two solutions at time step $n$, $ \bfw_h $ and $ \bfW_h $.  Plugging them into \eqref{FDLM_EMAC},
setting $\bfe= \bfW_h - \bfw_h$ and subtracting their equations gives
\beas
\frac{1}{\Delta t}(\bfe,\bfv_h) + \Rey\inv(\nabla \bfe,\nabla \bfv_h) + (C_S \delta)^r \lip{ \norm{ \Grad \bfW_h }_{F}^s \Grad \bfW_h - \norm{\bfw_h}_F^s \Grad \bfw_h }{\Grad \bfv_h } \\
= -c( \bfW_h,\bfe,\bfv_h) - c(\bfe, \bfw_h,\bfv_h), \quad \forall \bfv_h \in X_h.
\eeas

Taking $\bfv_h=\bfe$ causes the first nonlinear term vanish and produces
\begin{align*}
\frac{1}{\Delta t}\| \bfe \|^2 + \Rey\inv\| \nabla \bfe \|^2 + (C_S \delta)^r \lip{ \norm{ \Grad \bfW_h }_{F}^s \Grad \bfW_h - \norm{\bfw_h}_F^s \Grad \bfw_h }{\Grad \bfe} \\
= -c( \bfW_h,\bfe,\bfe) - c(\bfe, \bfw_h,\bfe)\\
= \frac{1}{2} c( \bfe,\bfW_h,\bfe) - c(\bfe, \bfw_h,\bfe)
\end{align*}
Using the stability bound  \eqref{Stability}, Lemmas \ref{C infty bound}, and \ref{lemma1} 
\begin{align*}
\frac{1}{\delt} \norm{\bfe}^2 + \Rey\inv \norm{ \Grad \bfe}^2 + C \norm{\Grad \bfe}_{s+2}^{s+2}
& \le M ( \norm{\bfW_h}_{\infty} + \| \bfw_h \|_{\infty} ) \| \nabla \bfe \| \| \bfe \| \\
& \leq C h^{-\frac{d}{2}}  ( \norm{\bfW_h} + \| \bfw_h \| ) \| \nabla \bfe \| \| \bfe \| \\
& \le C  h^{-\frac{d}{2}}\| \nabla \bfe \| \| \bfe \|.
\end{align*}
Dropping the second and third term on the left-hand side then using Lemma \ref{lem: Inverseineq}, the inverse inequality, yields
\bea
\delt\inv \norm{ \bfe }^2 &\le& C h^{-1-\frac{d}{2}} \| \bfe \|^2,
\eea

which implies $\norm{\bfe}=0$ and thus uniqueness of solutions under the condition that $\Delta t < O(h^{1+\frac{d}{2}})$  is satisfied. The proof for the SKEW formulation can be found in Lemma 6 in \cite{ReyesGSM}, giving a slightly worse uniqueness condition of $ \delt \leq O(h^3)$.
\end{proof}

\section{Convergence Analysis} \label{erroranalysis}

An NSE variational formulation can be stated as: 
Find $\bfu \in L^2(0, T; X) \cap L^{\infty}(0, T; L^2(\Omega))$, $p \in L^2(0,T; Q) $ with $\bfu_t \in L^2(0,T;X')$ satisfying
\begin{align}
\begin{aligned}
\label{var sol NS} 
        (\bfu_t,\bfv)+(\bfu \cdot \nabla \bfu,\bfv)+\Rey^{-1}(\nabla \bfu, \nabla \bfv)-(p, \nabla \cdot \bfv) = (\bff,\bfv), \quad \forall \bfv \in X, \\
    (\nabla \cdot \bfu, q) = 0, \quad \forall q \in Q.
\end{aligned}
\end{align}
While the variational form of the LM model can be stated as: 
Find $\bfw \in L^2(0, T; X_s) \cap L^{\infty}(0, T; L^{s+2}(\Omega))$, $p \in L^2(0,T; Q) $ with $\bfw_t \in L^2(0,T;X_s')$ satisfying
\begin{align}
\begin{aligned}
\label{varLM} \hspace{- 1.75mm} 
        (\bfw_t,\bfv)+(\bfw \cdot \nabla \bfw,\bfv)+\Rey^{-1}(\nabla \bfw, \nabla \bfv)+(C_S \delta)^r& ( \| \nabla \bfw  \|_F^s  \nabla \bfw, \nabla \bfv) \\
        -(p, \nabla \cdot \bfv) &= (\bff,\bfv), \:\forall \bfv \in X_s, \\
    (\nabla \cdot \bfw, q)&= 0, \: \forall q \in Q.
\end{aligned}
\end{align}

Since LM equation is an approximation to NSE, the error between these two equations is stated in the following lemma.
\begin{Lemma}[\cite{ReyesGSM} Continuous Error between NSE and LM] \label{NS proof lemma}
With $\bfu$ being smooth enough satisfying \eqref{var sol NS} and $\bfw$ satisfying \eqref{varLM}, we have constants $C_1,C_2 > 0$ such that
\begin{multline*}
    \vert \vert \bfu - \bfw \vert \vert^2 + \frac{\Rey^{-1}}{2} \int_0^t \vert \vert \nabla (\bfu-\bfw) \vert \vert^2 d \tau + C (C_S \delta)^r \int_0^t \vert \vert \nabla (\bfu-\bfw) \vert \vert_{s+2}^{s+2} d \tau \\
    \leq 
     C_2\exp\pare{C_1 \Rey^3 \int_0^{\tau} \norm{ \nabla \bfu}^4 ds } (C_S \delta)^{2r} \int_0^t \vert \vert \nabla \bfu \vert \vert^2 d \tau .
\end{multline*}
\end{Lemma}
To conduct the numerical error estimate, we will assume that $(\bfw, P)$ satisfies the following regularity requirement:
\begin{equation}\label{1st:regularity}
\left\{
\begin{aligned}
&\bfw\in L^{\infty}(0,T;(L^2(\Omega))^d),\bfw\in H^{1}(0,T;(H^{k+1}(\Omega))^d),\\
&\;\bfw_{tt} \in L^2(0,T;(H^1(\Omega))^d),  \;\bfw_{ttt} \in L^2(0,T;(L^2(\Omega))^d),\\
&P\in L^{2}(0,T;H^k(\Omega)). 
\end{aligned}
\right.
\end{equation}

\begin{Theorem}[\cite{ReyesGSM} Error for SKEW LM model]
\label{LM SKEW Error}
If $ (\bfw, p)$ be the solution of the SKEW formulation of the LM satisfying no-slip boundary conditions and the regularity assumption \eqref{1st:regularity}. Let $\bfw_h^n , n = 0, 1, \dots , M-1 $ be the finite element solution given by \eqref{DivfreeFDLM} using $(P_k,P_{k-1}) (k \geq 2 )$ elements. Then, the numerical error satisfies 
\bea
\trinorm{ \bfw(T) - \bfw_h^M }_{\infty,0}^2 &+& \Rey^{-1}\delt \sum_{n=0}^{M-1} \left\|\Grad \left( \bfw(t\chalf) - \bfw\chalf_h \right)  \right\|^{2} \nonumber \\ 
&+&  C (C_S \delta)^r \delt \sum_{n=0}^{M-1} \norm{\Grad \left( \bfw(t\chalf) - \bfw\chalf_h \right) }_{s+2}^{s+2} \nonumber \\
& = & CKO ( h^{2k} + \delta^{2r} h^{2k} + (\Delta t)^4 + \delta^{2r} (\Delta t)^4   ),
\eea
where $ K = \exp{\pare{ \delt \sum_{n=0}^{l} \frac{\gamma_n}{1-\delt \gamma_n}}} $ and $ \gamma_n = C \pare{ \norm{\Grad \bfw\chalf}_{\infty} + \Rey \norm{\bfw\chalf}_{\infty}^2 }$
\end{Theorem}

\begin{Theorem}[ Error for EMAC LM model]
\label{LM EMAC Error}
If $ (\bfw, P)$ be the solution of the EMAC formulation of the LM satisfying no-slip boundary conditions and the regularity assumption \eqref{1st:regularity}. Let $\bfw_h^n , n = 0, 1, \dots , M-1 $ be the finite element solution given by \eqref{DivfreeFDLM_EMAC} using $(P_k,P_{k-1}) (k \geq 2 )$ elements. Then, the numerical error satisfies 
\bea
\trinorm{ \bfw(T) - \bfw_h^M }_{\infty,0}^2 &+& \Rey^{-1}\delt \sum_{n=0}^{M-1} \left\|\Grad \left( \bfw(t\chalf) - \bfw\chalf_h \right)  \right\|^{2} \nonumber \\ 
&+&  C (C_S \delta)^r \delt \sum_{n=0}^{M-1} \norm{\Grad \left( \bfw(t\chalf) - \bfw\chalf_h \right) }_{s+2}^{s+2} \nonumber \\
& = & CKO ( h^{2k} + \delta^{2r} h^{2k} + (\Delta t)^4 + \delta^{2r} (\Delta t)^4   ),
\eea
where $ K = \exp{\pare{ \delt \sum_{n=0}^{l} \frac{\gamma_n}{1-\delt \gamma_n}}} $ and $ \gamma_n = C \pare{\norm{\Grad\bfw\chalf}_{\infty}+1}$ 
\end{Theorem}

\begin{Remark}
    Comparing both Theorems \ref{LM SKEW Error} and \ref{LM EMAC Error} the Key improvement is seen in the Gronwall constant K. In particular for the EMAC formulation we note no explicit dependence on Reynolds numbers. In contrast we see for the SKEW formulation under the same smoothness assumptions on the true solution we have a Gronwall constant of $ \gamma_n = O( \Rey ) $. This suggests that the EMAC-LM has better longer time accuracy than the corresponding SKEW-LM. 
\end{Remark}

\begin{proof}

We begin by evaluating the weak formulation of the LM \eqref{varLM} at $t = t\chalf$ to get
\bea
  \left(\frac{\bfw^{n+1}-\bfw^n}{\delt}, \bfv_h\right) +  \Rey^{-1} \left(\Grad  \bfw\chalf,\Grad \bfv_h\right)  +  c \left( \bfw\chalf, \bfw\chalf , \bfv_h\right) \nonumber \\ 
 - \left( P(t\chalf) , \Grad \cdot \bfv_h \right) + (C_S \delta)^r \left( \norm{ \Grad \bfw\chalf }_F^s \Grad \bfw\chalf , \Grad \bfv_h  \right) \nonumber  \\
 = \left(\bff (t\chalf), \bfv_h\right) \, + \, D\left(\bfw, \bfv_h\right), \quad \forall \bfv_{h} \in V_{h} ,  \label{TRUE} 
\eea
where
\begin{align}
 &D(\bfw,\bfv_h) = \Rey^{-1}  \left(\Grad  \left(\bfw\chalf - \bfw (t\chalf) \right) ,\Grad \bfv_h\right)  + c \left( \bfw\chalf , \bfw\chalf , \bfv_h\right) \nonumber \\
 +& \left(\frac{\bfw^{n+1} - \bfw^{n}}{\delt} - \bfw_{t}(t\chalf),\bfv_h\right) - c \left(\bfw(t\chalf), \bfw ( t\chalf ), \bfv_h\right)    \nonumber \\
 +& (C_S \delta)^r \left( \norm{ \Grad \bfw\chalf }_F^s \Grad \bfw\chalf - \norm{ \Grad \bfw (t\chalf) }_F^s \Grad \bfw (t\chalf)  , \Grad \bfv_h  \right).
 \label{TAIL}
\end{align}
Denote $\bfe^{n} = \bfw^n - \bfw^{n}_{h}$ and subtract \eqref{DivfreeFDLM_EMAC} from \eqref{TRUE} to get the error equation
\bea
\frac{1}{\delt} \left( \bfe^{n+1} - \bfe^{n}, \bfv_h\right) + c \left( \bfe\chalf, \bfw\chalf , \bfv_h\right)  +  c \left( \bfw\chalf_h, \bfe\chalf  ,\bfv_h\right)  \nonumber \\
 \textstyle
 + \Rey^{-1} \left(\Grad \bfe\chalf, \Grad \bfv_h\right) + (C_S \delta)^r \pare{ \norm{ \nabla\bfw\chalf }_F^s  \nabla\bfw\chalf , \nabla \bfv_h } \nonumber \\
  -  (C_S \delta)^r \left( \norm{ \Grad\bfw\chalf_h }_F^s  \nabla\bfw\chalf_h , \nabla \bfv_h \right)  = \left( P(t\chalf) , \Grad \cdot \bfv_h \right) + D\left(\bfw, \bfv_h\right) . \: \label{eq:error37}
\eea
Next, we decompose the velocity error as 
\beas
 \bfe^{n} \, = \, \left(\bfw^n - \P_{st} \left(\bfw^n\right)\right) \, - \, \left(\bfw_h^n - \P_{st} \left(\bfw^{n}\right)\right) \, = \, \bfeta^{n} - \bfphi^{n}_{h},
\eeas
where $ P_{st} $ is the Stokes projection given by $ (\Grad \P_{st}(\bfw) - \Grad \bfw,\Grad \bfv) =0$, \: $\forall \bfv \in V_h$. With the choice  $\bfv_{h} = \bfphi\chalf_h$,  and using $ (q_h, \Grad \cdot \bfphi_h\chalf) =0, \: \forall  q_h \in Q_h, $ we get the following error equation
\bea
  \frac{\norm{ \bfphi^{n+1}_h }^2 - \norm{ \bfphi^{n}_h }^{2} }{2\delt}  +  \Rey^{-1} \left\|\Grad \bfphi\chalf_h\right\|^{2} +  (C_S \delta)^r \pare{ \norm{ \Grad\bfw\chalf_h }_F^s  \nabla\bfw\chalf_h , \nabla \bfphi\chalf_h}\nonumber \\
 =  \left( \frac{\bfeta^{n+1} - \bfeta^{n}}{\Delta t}, \bfphi\chalf_h\right)  + (C_S \delta)^r ( \Vert \nabla\bfw\chalf \Vert_F^s  \nabla\bfw\chalf , \nabla \bfphi\chalf_h ) \nonumber \\
   - c \left( \bfeta\chalf , \bfw\chalf , \bfphi\chalf_h\right)  - c \left( \bfw\chalf_h , \bfeta\chalf, \bfphi\chalf_h\right) \,
 + c \left(\bfphi\chalf_h , \bfw\chalf , \bfphi\chalf_h\right)
  \nonumber \\ 
  + c \pare{ \bfw\chalf_h, \bfphi\chalf_h, \bfphi\chalf_h} - \left( P(t\chalf) - q_h, \Grad \cdot \bfphi\chalf_h \right)  - \, D \left(\bfw, \bfphi\chalf_h\right). 
\eea
Note that $ \Rey\inv \left(\Grad \bfeta\chalf, \Grad \bfphi\chalf_h\right) =0 $ by definition of the projection $ P_{st}$.

In contrast from the SKEW formulation, the term $\textstyle c(\bfw\chalf_h, \bfphi\chalf_h, \bfphi\chalf_h) \neq 0 $. We subtract the term $(C_S \delta)^r ( \Vert \nabla \P_{st}(\bfw\chalf) \Vert_F^s  \nabla \P_{st}(\bfw\chalf) , \nabla \bfphi_h\chalf)$ 
from both sides as well as inserting $ \pm \; c(P_{st}(\bfw\chalf),\bfphi\chalf_h, \bfphi\chalf_h)$ and\\ $\pm c(P_{st}(\bfw\chalf),\bfeta\chalf ,\phi_h\chalf )$ on the RHS and simplifiing gives
\bea
\label{EQ5}
 & &\frac{ \left\|\bfphi^{n+1}_h \right\|^{2} - \left\|\bfphi^{n}_h\right\|^{2}}{2\delt} +  \Rey^{-1} \left\|\Grad \bfphi\chalf_h \right\|^{2}  +  (C_S \delta)^r \pare{ \Vert \nabla \bfw\chalf_h \Vert_F^s  \nabla \bfw\chalf_h , \nabla \bfphi\chalf_h } \nonumber \\ 
 & &- (C_S \delta)^r \pare{ \Vert \nabla P_{st}(\bfw\chalf) \Vert_F^s  \nabla P_{st}(\bfw\chalf) , \nabla \bfphi\chalf_h}
 =  \left( \frac{\bfeta^{n+1} - \bfeta^{n}}{\Delta t}, \bfphi\chalf_h\right)   \nonumber \\
 &  & -  c \pare{ \bfeta\chalf,\bfw\chalf, \phi_h\chalf }
 + c \pare{ \bfphi\chalf_h , P_{st}(\bfw\chalf) , \bfphi\chalf_h } + c \pare{ \bfphi\chalf_h,\bfphi\chalf_h,\bfphi\chalf_h}
  \nonumber \\
 & & -c \pare{P_{st}(\bfw\chalf),\bfeta\chalf,\bfphi_h\chalf}  + c\pare{ P_{st}(\bfw\chalf) ,\bfphi\chalf_h,\bfphi\chalf_h } \nonumber \\
 & &+ (C_S \delta)^r \pare{ \Vert \nabla \bfw\chalf \Vert_F^s  \nabla \bfw\chalf , \nabla \bfphi\chalf_h } - \left( P(t\chalf) - q_h, \Grad \cdot \bfphi\chalf_h \right) \nonumber\\
 & &-  (C_S \delta)^r \pare{ \Vert \nabla P_{st}(\bfw\chalf) \Vert_F^s  \nabla P_{st}(\bfw\chalf) , \nabla \bfphi\chalf_h } -  D \left( \bfw, \bfphi\chalf_h \right) .
\eea
Based on Lemma \ref{lemma1}
\bea
 (C_S \delta)^r ( \Vert \nabla \bfw\chalf_h \Vert_F^s  \nabla \bfw\chalf_h , \nabla \bfphi\chalf_h ) - (C_S \delta)^r ( \Vert \nabla P_{st}(\bfw\chalf) \Vert_F^s  \nabla P_{st}(\bfw\chalf) , \nabla \bfphi\chalf_h)  \nonumber \\ 
 \geq C (C_S \delta)^r \norm{ \Grad (\bfw\chalf_h - P_{st}(\bfw\chalf))}_{s+2}^{s+2}  = C (C_S \delta)^r \norm{\Grad \bfphi\chalf_h}_{s+2}^{s+2}.
\eea
 Rewriting \eqref{EQ5} noting $c(\bfu,\bfu,\bfu) =0 $ and using \eqref{eq: emac_id} gives
\begin{flalign}
\label{EQ6}
&  \frac{1}{2 \delt} \left(\left\|\bfphi^{n+1}_h\right\|^{2} - \left\|\bfphi^{n}_h\right\|^{2}\right) +  \Rey^{-1} \left\|\Grad \bfphi\chalf_h\right\|^{2}  + C (C_S \delta)^r \norm{\Grad \bfphi_h\chalf}_{s+2}^{s+2} \qquad \qquad
  \nonumber &&\\
 &\leq \abs{ c \left( \bfeta\chalf , \bfw\chalf , \bfphi\chalf_h\right) }
 + \abs{  c \left( P_{st}(\bfw\chalf) , \bfeta\chalf, \bfphi\chalf_h\right) } \nonumber &&\\
& + \frac{1}{2}\abs{ c \pare{ P_{st}( \bfw\chalf) ,\bfphi\chalf_h,\bfphi\chalf_h}} + \abs{ \pare{ \frac{\bfeta^{n+1}-\bfeta^{n}}{\delt} ,\bfphi\chalf_h } } \nonumber &&\\
& + \Big| (C_S \delta)^r ( \Vert \nabla \bfw\chalf \Vert_F^s  \nabla \bfw\chalf , \nabla \bfphi\chalf_h ) \nonumber &&\\
& -  (C_S \delta)^r ( \Vert \nabla P_{st}(\bfw\chalf) \Vert_F^s \nabla P_{st}(\bfw\chalf) , \nabla \bfphi\chalf_h) \Big|  \nonumber \\
& + \abs{ \left( P(t\chalf) - q_h, \Grad \cdot \bfphi\chalf_h \right) }  + \,\abs{ D \left(\bfw, \bfphi\chalf_h\right)}. &&
\end{flalign}
Next we bound each term on the right hand side using Lemmas \ref{C_TRIL}, \ref{C infty bound} , Poincar\'{e} inequality, Cauchy-Schwartz inequality and Young's inequality.

\bea\label{Bound1}
     \pare{ \frac{\bfeta^{n+1} - \bfeta^{n}}{\Delta t}, \bfphi\chalf_h} &\leq&\frac{1}{4}\norm{\frac{\bfeta^{n+1}-\bfeta^{n}}{\delt}}^2+\norm{\bfphi_h\chalf}^2\nonumber\\
    &\leq&\frac{1}{4}\int_\Omega\pare{\frac{1}{\delt}\int_{t^{n}}^{t^{n+1}}|\bfeta_t|dt}^2dx+\norm{\bfphi_h\chalf}^2\nonumber\\
    &\leq&\frac{1}{4\delt} \int_{t^{n}}^{t^{n+1}} \norm{\bfeta_t}^2 dt + \norm{\bfphi_h\chalf}^2.
\eea

\bea
 \abs{\left(P(t\chalf) - q_{h}, \Grad \cdot \bfphi\chalf_h\right)  }\leq  \frac{\Rey\inv}{16} \left\|\Grad \bfphi\chalf_h \right\|^{2}  +  C\Rey \inf\limits_{  q_{h} \in Q_{h}} \norm{P - q_h}^2, \label{BOUND2}
\eea
\bea
&  & \abs{c \left( \bfeta\chalf , \bfw\chalf , \bfphi\chalf_h \right)}
\leq  \norm{\Grad \bfeta\chalf} \norm{\bfw\chalf}^{\frac12} \norm{\Grad\bfw\chalf}^{\frac12} \norm{\Grad \bfphi\chalf_h} \nonumber \\
&\leq& \frac{\Rey\inv}{16} \norm{ \Grad \bfphi\chalf_h }^2 + C\Rey \norm{\bfw\chalf} \norm{\Grad \bfw\chalf} \norm{\Grad \bfeta \chalf}^2 ,\label{Bound3} \qquad 
\eea
\bea
\abs{c \left( P_{st}(\bfw\chalf) , \bfeta\chalf , \bfphi\chalf_h \right)}
\leq C \norm{  \Grad P_{st}(\bfw\chalf)} \norm{ \bfeta\chalf}^{\frac12} \norm{\Grad \bfeta\chalf}^{\frac12} \norm{\Grad \bfphi\chalf_h } \nonumber \\
\leq  \frac{\Rey\inv}{16} \norm{ \Grad \bfphi\chalf_h }^2 + C \Rey \norm{\Grad \bfw\chalf}^2 \norm{\bfeta\chalf}  \norm{\Grad \bfeta\chalf},\label{Bound4} 
\qquad
\eea
and
\bea
\abs{ c \pare{ P_{st}( \bfw\chalf) ,\bfphi\chalf_h,\bfphi\chalf_h}} 
&\leq& C \norm{ \Grad P_{st}(\bfw\chalf) }_{\infty} \norm{\bfphi\chalf_h}^2\nonumber \\
&\leq& C \norm{ \Grad \bfw\chalf }_{\infty} \norm{\bfphi\chalf_h}^2. \label{Bound6}
\qquad
\eea

The Smagorinsky term is bounded using Lemma \ref{dufunzburgerlemma}
\bea
\abs{ (C_S \delta)^r ( \Vert \nabla \bfw\chalf \Vert_F^s  \nabla \bfw\chalf , \nabla \bfphi_h\chalf ) -  (C_S \delta)^r ( \Vert \nabla P_{st}(\bfw\chalf) \Vert_F^s  \nabla P_{st}(\bfw\chalf) , \nabla \bfphi\chalf_h) }\nonumber\\
\leq C (C_S \delta)^r \norm{\Grad \bfeta\chalf} \norm{\Grad \bfphi\chalf_h}
\leq  \frac{\Rey\inv}{16} \left\|\Grad \bfphi\chalf_h\right\|^{2} +  (C (C_S \delta)^r)^2 \Rey \left\|\Grad \bfeta\chalf\right\|^{2}, \label{Bound7} \qquad
\eea

where we note that  C depends on $\norm{\Grad \bfw}_{\infty}$, i.e., $\bfw \in W^{1,\infty}$.

Combining \eqref{Bound1} - \eqref{Bound7} results in
\bea
& &\frac{1}{2 \delt} \left(\left\|\bfphi^{n+1}_h\right\|^{2} - \left\|\bfphi^{n}_h\right\|^{2}\right) +  \frac{3\Rey^{-1}}{4} \left\|\Grad \bfphi\chalf_h\right\|^{2} + C (C_S \delta)^r \norm{\Grad \bfphi\chalf_h}_{s+2}^{s+2} \nonumber \\
&\leq&  C \pare{\norm{\Grad \bfw\chalf}_{\infty}+1}\norm{\bfphi\chalf_h}^2  + C \Rey \norm{\bfw\chalf} \norm{\Grad \bfw\chalf} \norm{\Grad \bfeta\chalf}^2 \nonumber \\
&+&  \, C \Rey \inf\limits_{q^{h} \in Q_{h}} \norm{P - q_h }^{2} + C \Rey \norm{\bfeta\chalf} \norm{\Grad \bfeta\chalf} \norm{\Grad \bfw\chalf}^2\nonumber \\
&+&  \frac{C}{\delt} \int_{t^n}^{t^{n+1}} \norm{ \bfeta_t}^{2} \; dt + \Rey C C_S^2 \delta^{2r} \norm{ \Grad \bfeta\chalf}^2 + \abs{D \left(\bfw, \bfphi\chalf_h\right)} .\qquad
\eea

We use Lemmas \ref{C_TRIL}, \ref{dufunzburgerlemma}, Appendix of \cite{ervin2007numerical} , Poincar\'{e} inequality, Cauchy-Schwartz inequality and Young's inequality to bound the tail $D \left(\bfw, \bfphi_{h}^{n+\frac{1}{2}}\right)$. \\
\bea
& & \left(\frac{\bfw(t^{n+1}) - \bfw(t^{n})}{\delt} - \bfw_{t}(t\chalf), \bfphi\chalf_h\right) \nonumber \\
 & \leq & \frac{\Rey\inv}{24} \left\|\Grad \bfphi\chalf_h \right\|^{2} \, +C \Rey  \left\|\frac{\bfw(t^{n+1}) - \bfw(t^{n})}{\delt} - \bfw_{t}(t\chalf)\right\|^{2} \nonumber \\
 & \leq & \frac{\Rey\inv}{24} \left\|\Grad \bfphi\chalf_h\right\|^{2} \, + C \Rey  \delt^{3} \int_{t^{n}}^{t^{n+1}} \left\|\bfw_{ttt} \right\|^{2} dt \label{BOUND10} \;,
\eea
\bea
& &\Rey^{-1}  \left(\Grad  \left(\bfw\chalf - \bfw (t\chalf) \right) ,\Grad \bfphi\chalf_h\right) \nonumber \\
&\leq& \frac{\Rey\inv}{24} \left\|\Grad \bfphi\chalf_h\right\|^{2} \, + C\Rey^{-1} \left\|\Grad \left( \bfw\chalf - \bfw(t\chalf)\right)\right\|^{2} \nonumber \\
 &\leq&  \frac{\Rey\inv}{24} \left\|\Grad \bfphi\chalf_h\right\|^{2} \, +  C \Rey^{-1} \delt^{3} \int_{t^{n}}^{t^{n+1}} \left\|\Grad \bfw_{tt}\right\|^{2} dt, \qquad \quad  \label{BOUND11}
\eea
and
\hspace{-2em}
\bea
& & c \left( \bfw\chalf , \bfw\chalf , \bfphi\chalf_h\right)  -  c \left(\bfw(t\chalf), \bfw ( t\chalf ), \bfphi\chalf_h\right) \nonumber \\
&=& c \left( \bfw\chalf - \bfw(t\chalf) , \bfw\chalf , \bfphi\chalf_h\right)  -  c \left(\bfw(t\chalf), \bfw\chalf  - \bfw(t\chalf), \bfphi\chalf_h\right) \nonumber \\
&\leq& C \norm{\Grad (\bfw\chalf - \bfw(t\chalf))} \norm{ \Grad \bfphi\chalf_h}\left(\norm{\Grad \bfw\chalf } + \norm{\Grad \bfw(t\chalf)}\right) \nonumber \\
&\leq& C \Rey \left( \norm{ \Grad \bfw\chalf}^2 + \norm{\Grad \bfw(t\chalf) }^2 \right) \frac{\delt^3}{48} \int_{t_n}^{t^{n+1}} \norm{\Grad \bfw_{tt}}^2 dt + \frac{\Rey^{-1}}{8} \norm{ \Grad \bfphi\chalf_h }^2 \nonumber \\
& \leq & C \Rey {\delt^3} \left( \int_{t_n}^{t^{n+1}} \norm{\Grad \bfw\chalf}^4  + \norm{\Grad \bfw_{tt}}^4 + \norm{ \Grad \bfw(t\chalf) }^4  \; dt \right) \nonumber \\
& &+ \frac{\Rey^{-1}}{8} \norm{ \Grad \bfphi\chalf_h }^2 \nonumber \\
& \leq & C \Rey {\delt^4} \left( \norm{ \Grad \bfw\chalf }^4 + \norm{ \bfw(t\chalf) }^4 \right) + C \Rey \delt^3  \int_{t_n}^{t^{n+1}} \norm{\bfw_{tt}}^4 \; dt \nonumber \\
& &+  \frac{\Rey^{-1}}{8} \norm{ \Grad \bfphi\chalf_h }^2. \qquad 
\eea

For the Smagorinsky term we use Lemma \ref{dufunzburgerlemma} along with Young's inequality
\bea
 & & (C_S \delta)^r \left( \norm{ \Grad \bfw\chalf }_F^s \Grad \bfw\chalf - \norm{ \Grad \bfw\chalf }_F^s \Grad \bfw (t\chalf)  , \Grad \bfphi\chalf_h  \right) \nonumber \\
 & \leq & C (C_S \delta)^r  \norm{\Grad\big(\bfw\chalf - \bfw(t\chalf)  \big)} \norm{\Grad\bfphi\chalf_h} \nonumber \\
 & \leq & \frac{\Rey\inv}{24} \norm{\Grad \bfphi\chalf_h}^2 + C  C_S^2  \delta^{2r} \Rey  \norm{ \Grad\big( \bfw\chalf - \bfw (t\chalf) \big) }^2 \nonumber \\
 & \leq & \frac{\Rey\inv}{24} \norm{\Grad \bfphi\chalf_h}^2 + C C_S^2 \delta^{2r}\Rey  \delt^3 \int_{t_n}^{t^{n+1}}  \norm{ \Grad \bfw_{tt} }^2 \; dt.
\eea
Where C  depends on $ \norm{\Grad \bfw}_\infty$. Combining the above bounds we arrive at
\bea
& &\frac{1}{2 \delt} \left(\left\|\bfphi^{n+1}_h\right\|^{2} - \left\|\bfphi^{n}_h\right\|^{2}\right) +  \frac{\Rey^{-1}}{2} \left\|\Grad \bfphi\chalf_h\right\|^{2} + C (C_S \delta)^r \norm{\Grad \bfphi^{n+ \frac{1}{2}}_h}_{s+2}^{s+2} \nonumber \\
&\leq&  C \pare{\norm{\Grad \bfw\chalf}_{\infty}+1}\norm{\bfphi\chalf_h}^2  + C \Rey \norm{\bfw\chalf} \norm{\Grad \bfw\chalf} \norm{\Grad \bfeta\chalf}^2 \nonumber \\
&+&  \, C \Rey \inf\limits_{q_{h} \in Q_{h}} \left\|P - q_{h}\right\|^{2} + C \Rey \norm{\bfeta\chalf} \norm{\Grad \bfeta\chalf} \norm{\Grad \bfw\chalf}^2 \nonumber \\
&+&  \frac{C}{\delt} \int_{t^n}^{t^{n+1}} \norm{ \bfeta_t}^{2} \; dt  +  \Rey C (C_S \delta)^{2r} \norm{ \Grad \bfeta\chalf}^2  +  C \Rey \delt^{3} \int_{t^{n}}^{t^{n+1}} \left\|\bfw_{ttt} \right\|^{2} dt \nonumber   \\
&+&  C \Rey^{-1} \delt^{3} \int_{t^{n}}^{t^{n+1}} \left\|\Grad \bfw_{tt}\right\|^{2} dt + C \Rey {\delt^4} \left( \norm{ \Grad \bfw\chalf }^4 + \norm{ \bfw\chalf}^4 \right) \nonumber \\
&+& C \Rey \delt^3  \int_{t_n}^{t^{n+1}} \norm{\bfw_{tt}}^4 \; dt + C \Rey (C_S\delta)^{2r} \delt^3 \int_{t_n}^{t^{n+1}}  \norm{ \Grad \bfw_{tt} }^2 \; dt.
\eea
Next, we sum over time steps and multiply by $2 \Delta t$, 
\bea
& &\left\| \bfphi_h^M \right\|^2 + \Rey^{-1}\delt \sum_{n=0}^{M-1} \left\|\Grad \bfphi\chalf_h\right\|^{2}
+  C (C_S \delta)^r \delt\sum_{n=0}^{M-1} \norm{\Grad \bfphi_h\chalf}_{s+2}^{s+2} \nonumber  \nonumber \\
&\leq& 
 C \delt\sum_{n=0}^{M-1}\pare{ \norm{\Grad \bfw\chalf}_{\infty} +1}\norm{\bfphi\chalf_h}^2 +  C \Rey \delt\sum_{n=0}^{M-1} \inf\limits_{q_{h} \in Q_{h}} \left\|P - q_{h}\right\|^{2}  \nonumber \\
&+& C \Rey \delt\sum_{n=0}^{M-1} \norm{\bfw\chalf} \norm{\Grad \bfw\chalf} \norm{\Grad \bfeta\chalf}^2 + C \int_{0}^{T} \norm{\bfeta_t}^2 \;dt  \nonumber \\
&+&   C \Rey \delt\sum_{n=0}^{M-1} \norm{\bfeta\chalf} \norm{\Grad \bfeta\chalf} \norm{\Grad \bfw\chalf}^2 + C \Rey \delt^4  \int_{0}^{T} \norm{\bfw_{tt}}^4 \; dt  \nonumber \\
&+&   \Rey C C_S^2 \delta^{2r} \delt\sum_{n=0}^{M-1} \norm{ \Grad \bfeta\chalf}^2  +  C \Rey \delt^{4} \int_{0}^{T} \left\|\bfw_{ttt} \right\|^{2} dt \nonumber   \\
&+&  C \Rey^{-1} \delt^{4} \int_{0}^{T} \left\|\Grad \bfw_{tt}\right\|^{2} dt +C \Rey C_S^2 \delta^{2r} \delt^4 \int_{0}^{T}  \norm{ \Grad \bfw_{tt} }^2 \; dt \nonumber \\
&+&  C \Rey {\delt^5} \sum_{n=0}^{M-1} \left( \norm{ \Grad \bfw\chalf }^4 + \norm{ \bfw\chalf}^4 \right).
\eea

We continue to estimate the following terms.
\bea 
C \int_{0}^{T} \norm{\bfeta_t}^2 \;dt  \leq C h^{2k+2} \int_{0}^{T} \norm{\bfw_t}_{k+1,2}^2 \;dt, 
\eea
\bea
 \delt \sum_{n=0}^{M-1} C \Rey \inf\limits_{q_{h} \in Q_{h}} \norm{P - q_h }^{2}  \leq   C \Rey  h^{2k} \delt \sum_{n=0}^{M}\norm{P}_{k,2}^{2} \leq  C \Rey  h^{2k} \trinorm{P}_{2,k}^{2} ,
\eea
\bea
    & & \delt \sum_{n=0}^{M-1}  C \Rey \norm{\Grad \bfeta\chalf} \norm{\bfeta\chalf}  \norm{\Grad \bfw\chalf}^2 \nonumber\\
    & \leq & C \Rey h^{2k+1}  \delt \sum_{n=0}^{M-1} \bigg( \norm{\bfw^{n+1} }_{k+1,2}^2\nonumber \\
    & & + \norm{\bfw^{n+1} }_{k+1,2} \norm{\bfw^{n} }_{k+1,2} + \norm{\bfw^{n+1} }_{k+1,2}^2 \bigg) \norm{ \Grad \bfw\chalf }^2  \nonumber\\
    & \leq & C \Rey h^{2k+1} \delt \pare{\sum_{n=0}^{M}||\bfw^{n}||^4_{k+1,2}+\sum_{n=0}^{M}||\nabla \bfw^{n}||^4},
\eea

\bea
    & & \delt \sum_{n=0}^{M-1}  C \Rey  \norm{\Grad \bfeta\chalf}^2 \norm{ \bfw\chalf }\norm{\Grad \bfw\chalf} \nonumber\\
    & \leq & C \Rey \delt \sum_{n=0}^{M-1} \pare{\norm{\Grad \bfeta^{n+1} }^2 + \norm{\Grad \bfeta^n}^2} \norm{\bfw\chalf} \norm{\Grad \bfw\chalf} \nonumber\\
    & \leq & C \Rey h^{2k+1} \delt \pare{\sum_{n=0}^{M}||\bfw^{n}||^4_{k+1,2}+\sum_{n=0}^{M}||\nabla \bfw^{n}||^4 +\sum_{n=0}^{M}|| \bfw^{n}||^4 } \nonumber,\\
    & \leq & C \Rey h^{2k+1}  \pare{ \trinorm{\bfw}^4_{4,k+1}+\trinorm{\nabla \bfw}^4_{4,0} +\trinorm{ \bfw}^4_{4,0}},
\eea
and
\bea
   \Rey C (C_S \delta)^{2r} \delt \sum_{n=0}^{M-1}  \norm{ \Grad \bfeta\chalf}^2 &\leq& 
    \Rey C (C_S \delta)^{2r} h^{2k}  \trinorm{ \bfw }_{2,k+1}^2.
\eea

Putting all the bounds together we obtain
\bea
& & \left\| \bfphi_h^M \right\|^2 + \Rey^{-1}\delt \sum_{n=0}^{M-1} \left\|\Grad \bfphi\chalf_h\right\|^{2}
+  C (C_S \delta)^r \delt\sum_{n=0}^{M-1} \norm{\Grad \bfphi_h\chalf}_{s+2}^{s+2} \nonumber \\
&\leq&  C \delt\sum_{n=0}^{M-1} \pare{\norm{\Grad \bfw\chalf}_{\infty}+1} \norm{\bfphi\chalf_h}^2   \nonumber\\
&+& C \Rey h^{2k+1} \pare{ \trinorm{ \bfw }_{{4,k+1}}^4 + \trinorm{ \Grad \bfw }_{{4,0}}^4 } + C \Rey h^{2k} \trinorm{P}_{2,k}^{2}  \nonumber \\
&+&  C h^{2k+2} \int_{0}^{T} \norm{\bfw_t}_{k+1,2}^2 \;dt + C \Rey h^{2k+1}  \pare{\trinorm{\bfw^{n}}^4_{4,k+1}+\trinorm{\nabla \bfw^{n}}^4_{4,0} + \trinorm{ \bfw^{n} }^4_{4,0} }  \nonumber \\
&+&  Ch^{2k}  \Rey (C_S \delta)^{2r}   \trinorm{ \bfw^{n} }_{k+1,2}^2  +  C \Rey \delt^{4} \int_{0}^{T} \left\|\bfw_{ttt} \right\|^{2} dt \nonumber   \\
&+&  C \Rey^{-1} \delt^{4} \int_{0}^{T} \left\|\Grad \bfw_{tt}\right\|^{2} dt + C \Rey {\delt^4} \left( \trinorm{ \Grad \bfw_{1/2} }_{4,0}^4 + \trinorm{ \bfw_{1/2}}_{4,0}^4 \right) \nonumber \\
&+& C \Rey \delt^4  \int_{0}^{T} \norm{\bfw_{tt}}^4 \; dt + C \Rey (C_S \delta)^{2r} \delt^4 \int_{0}^{T}  \norm{ \Grad \bfw_{tt} }^2 \; dt.
\eea

Thus, with $ \delt $ sufficiently small ($ \gamma_n \delt := C (\norm{\Grad \bfw^n }_{\infty}+1) \delt < 1 $) from Gronwall's Inequality (Lemma \ref{discreteGronwall})  we have the following

\bea
& & \left\| \bfphi_h^M \right\|^2 + \Rey^{-1}\delt \sum_{n=0}^{M-1} \left\|\Grad \bfphi\chalf_h\right\|^{2}
+  C (C_S \delta)^r \delt\sum_{n=0}^{M-1} \norm{\Grad \bfphi_h\chalf}_{s+2}^{s+2} \nonumber \\
&\leq& C \Rey h^{2k+1} \pare{ \trinorm{ \bfw }_{{4,k+1}}^4 + \trinorm{ \Grad \bfw }_{{4,0}}^4 } + C \Rey h^{2k} \trinorm{P}_{2,k}^{2}  \nonumber \\
&+&  C h^{2k+2} \int_{0}^{T} \norm{\bfw_t}_{2,k+1}^2 \;dt + C \Rey h^{2k+1}  \pare{\trinorm{\bfw^{n}}^4_{4,k+1}+\trinorm{\nabla \bfw^{n}}^4_{4,0} + \trinorm{ \bfw^{n} }^4_{4,0} }  \nonumber \\
&+&  Ch^{2k}  \Rey (C_S \delta)^{2r}   \trinorm{ \bfw^{n} }_{2,k+1}^2  +  C \Rey \delt^{4} \int_{0}^{T} \left\|\bfw_{ttt} \right\|^{2} dt \nonumber   \\
&+&  C \Rey^{-1} \delt^{4} \int_{0}^{T} \left\|\Grad \bfw_{tt}\right\|^{2} dt + C \Rey {\delt^4} \left( \trinorm{ \Grad \bfw_{1/2} }_{4,0}^4 + \trinorm{ \bfw_{1/2}}_{4,0}^4 \right) \nonumber \\
&+& C \Rey \delt^4  \int_{0}^{T} \norm{\bfw_{tt}}^4 \; dt + C \Rey (C_S \delta)^{2r} \delt^4 \int_{0}^{T}  \norm{ \Grad \bfw_{tt} }^2 \; dt.
\eea
\end{proof}

Combining the estimates from Lemma \ref{NS proof lemma} and Theorem \ref{LM EMAC Error} we have the following corollary 
\begin{Corollary}
\label{corollary result}
Let $ \bfu $ and $ \bfw$ be smooth enough solution and under the assumptions of Lemma \ref{NS proof lemma} and Theorem \ref{LM EMAC Error} we have

\beas
& & \trinorm{ \bfu - \bfw_h }_{\infty,0}^2 
 =   O(  h^{2k} + \delta^{2r} h^{2k}  + \delt^4 + \delta^{2r} \delt^4  + \delta^{2r} ).
\eeas

\end{Corollary}

\section{Numerical Experiments} \label{sec-numres}

In this section, experimental evaluation is presented for two benchmark problems.

\subsection{Taylor Green Vortex Problem}
A common benchmark problem for obtaining convergence rates of models is the Taylor Green Vortex problem which was first studied in \citep{taylor1923lxxv}. Knowing the true solution will allow investigation of convergence rates. In the unit square $\Omega=(0,1)\times(0,1)$, the true solution of the Taylor Green Vortex problem is 
\begin{gather*}
\begin{aligned}
    u_1(x,y,t) &= - \sin(\omega\pi y)\cos(\omega\pi x)\exp(-2\omega^2\Rey\inv\pi^2t),\\
    u_2(x,y,t) &= \cos(\omega\pi y)\sin(\omega\pi x)\exp{(-2\omega^2\Rey\inv\pi^2t)},\\
    p(x,y,t) &= \frac{- \cos(2\omega\pi y) - \cos(2\omega\pi x)}{4}\exp({-2\omega^2\Rey\inv\pi^2 t}).
\end{aligned}
\end{gather*}

This gives a series of vortices in an  $ \omega \times \omega $ array which decay as $ t \to \infty $. A uniform triangular mesh with $ m $ subdivions along each edge was used. We began with a mesh of $ m=16 $ and further refined by doubling until $ m =  96 $. {Following \cite{ReyesGSM}, $ r= \tfrac{4}{3}+s $ is chosen.} The model parameter $ \omega = 1 $, final time $ T = 0.1$, step size $ \Delta t = 0.0005$, and viscosity $ \Rey=100$  were chosen. Initial condition is $ \bfu_0 = \bfu(0) $, and we implemented Dirichlet boundary conditions. The Smagorinsky constant $ C_S = 0.01 $, and filter width $\delta =\frac{1}{m} $ were used for LM-EMAC simulation. As shown in Table \ref{tab-conv:m1}, LM-EMAC is second-order accurate.

\begin{table}[!ht]\small
	\centering
\begin{tabular}{l|l|l|l|l|l|l}
	\hline
	\hline \multicolumn{7}{c}{$\left\|{\bfu}-{\bfw}_h\right\|_{\ell^2\left(0, T ; H^1\right)}$ error} \\
	\hline $m$ & LM (s = 1) & Rate & LM (s=2) & Rate & LM (s=3) & Rate \\\hline
	 16 & $9.8226 \times 10^{-3}$ & - &         $9.8940 \times 10^{-3}$ & -              & $9.9128 \times 10^{-3}$ & - \\
	 32 & $1.5500\times 10^{-3}$   & 2.66 & $1.5518 \times 10^{-3}$ & 2.67          & $1.5520 \times 10^{-3}$ & 2.68\\
	 48 & $5.5738 \times 10^{-4}$ & 2.52 & $5.5741 \times 10^{-4}$ & 2.53          & $5.5743 \times 10^{-4}$ & 2.53 \\
	 64 & $2.8188 \times 10^{-4}$ & 2.37 & $2.8182 \times 10^{-4}$ & 2.37           & $2.8182 \times 10^{-4}$ & 2.37 \\
	 80 & $1.7006 \times 10^{-4}$ & 2.26 & $1.7001 \times 10^{-4}$ & 2.26            & $1.7002 \times 10^{-4}$ & 2.27 \\
	 96 & $1.1398 \times 10^{-4}$ & 2.19  & $1.1395 \times 10^{-4}$ & 2.19            & $1.1395 \times 10^{-4}$ & 2.19 \\
	\hline\hline
\end{tabular}
	\caption{Numerical results on the errors of  simulated velocity   at the terminal time $T = 0.1$ produced by the LM-EMAC finite element scheme.}
\label{tab-conv:m1}
\end{table}

\subsection{The channel flow past a forward-backward facing step}

We next test  the problem of the  channel flow passing a forward-backward facing step. 
 The domain of the step problem is a $40\times 10$ channel with a $1\times 1$ step five units into the channel at the bottom. 
No body force is imposed, i.e., $\bm f=0$. In addition, the boundary condition is given by no-slip condition on the top and bottom walls and the step, no-flux condition of the outflow in the right side, and a constant-in-time parabolic inflow with max inlet velocity of 1 in the left side, that is
\begin{align*}
	\bm u(0,y,t) = (y(10-y)/25, 0),\quad 0\leq y\leq 10.
\end{align*}
The initial velocity is set to be $\bm u_0=0$,  $\Rey=10^4$, and the terminal time $T=40$. We generate a spatial mesh with 6,010 vertices and 11,598 triangles for the domain, and choose a relatively large time step size $\tau=0.01$.

For this challenging large Reynolds number problem, we set the Smagorinsky constant $ C_S = 0.01 $, and filter width $ \delta = 0.01$ were used for both the LM-EMAC and LM-SKEW simulations. The LM simulation was run with parameter $ s = 1 $ and $ r= \tfrac{4}{3}+s$. 
The plots of the magnitude  of the simulated velocity at times $t$ = 10, 20, 30 and 40 are shown in Figures \ref{step:mag} and \ref{step:mag:skew} for LM-EMAC and LM-SKEW schemes, respectively.  It is well-known that the eddies will form behind the step,  and gradually grow and detach from the step \cite{manica2011enabling} in this problem. Such phenomenon can be clearly observed from the simulation results produced by the proposed LM-EMAC and LM-SKEW schemes. However, LM-EMAC is more stable than LM-SKEW if we check the behavior around the right-hand side boundary. This is due to the momentum and angular momentum-preserving property of EMAC scheme. To demonstrate such guess, we also plot the evolutions of energy, momentum and angular momentum generated by these two schemes as shown in Figure \ref{step:quan:com}. Since we prescribe the parabolic inflow, the energy is not dissipative. Instead, it is increasing along with the time. As illustrated in Figure \ref{step:quan:com}, the energy generated by LM-EMAC is smaller than LM-SKEW after $t=10$. Next, LM-EMAC preserves the momentum pretty well since its plot lies in between 266.6 and 266.8. However, the momentum of LM-SKEW oscillates between 265.4 and 267.2. Moreover, its peak tends to increase. Similar behaviors can be observed from the angular momentum plot. All three plots imply that LM-EMAC is more feasible for a longer time simulation.

\begin{figure}[!htbp]
	\centerline{
		\includegraphics[width=0.44\textwidth]{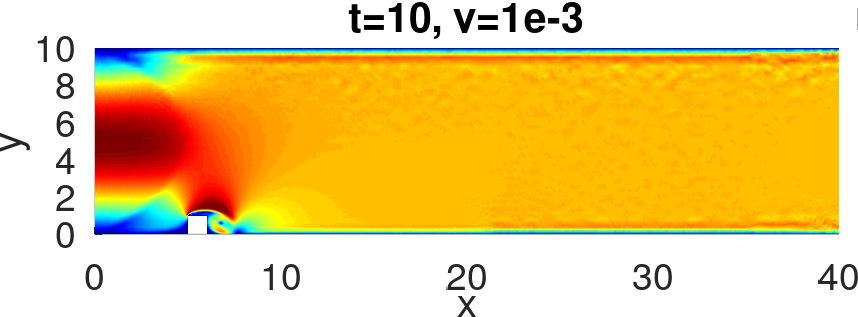}
		\includegraphics[width=0.44\textwidth]{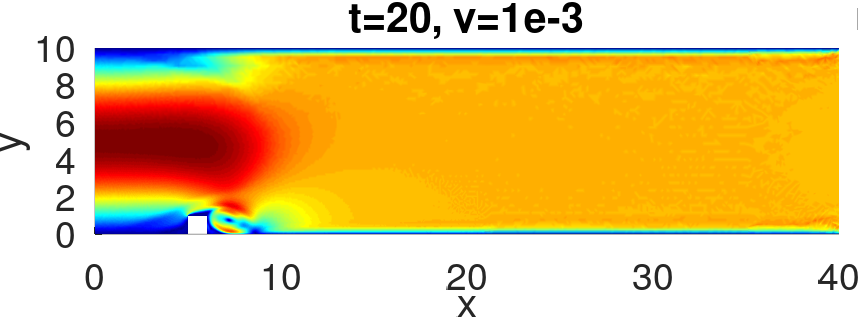}
		\includegraphics[width=0.1\textwidth]{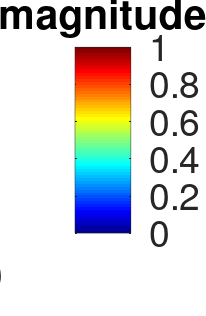}
	}\vspace{0ex}
	\centerline{
		\includegraphics[width=0.44\textwidth]{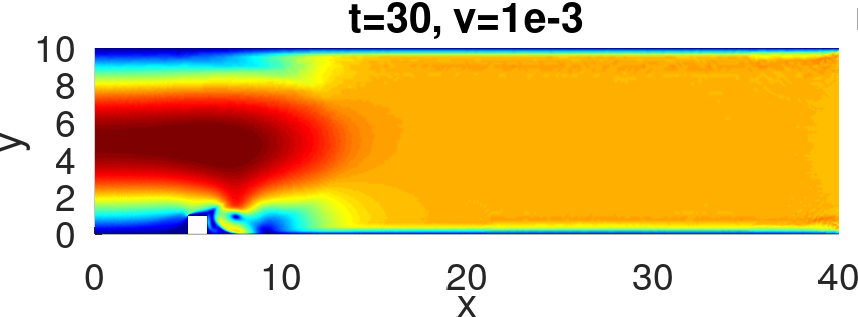}
		\includegraphics[width=0.44\textwidth]{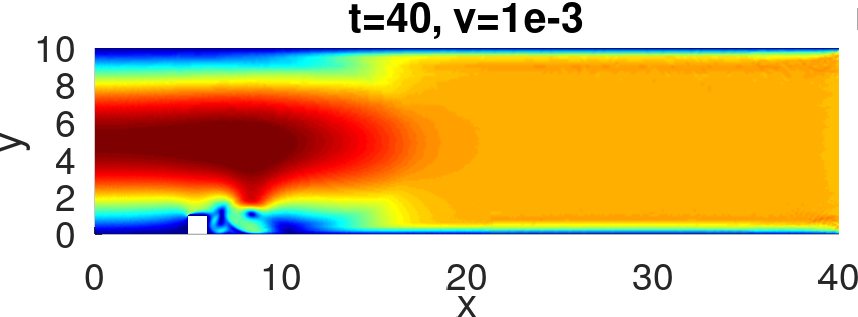}
		\includegraphics[width=0.1\textwidth]{colorbar.png}
	}
	\caption{Plots of magnitude of the simulated velocity at times $t$ = 10, 20, 30 and 40 generated by LM-EMAC scheme  for the channel flow past a forward-backward facing step problem with the viscosity $\Rey=10^4$.}
	\label{step:mag}
\end{figure}

\begin{figure}[!htbp]
	\centerline{
		\includegraphics[width=0.44\textwidth]{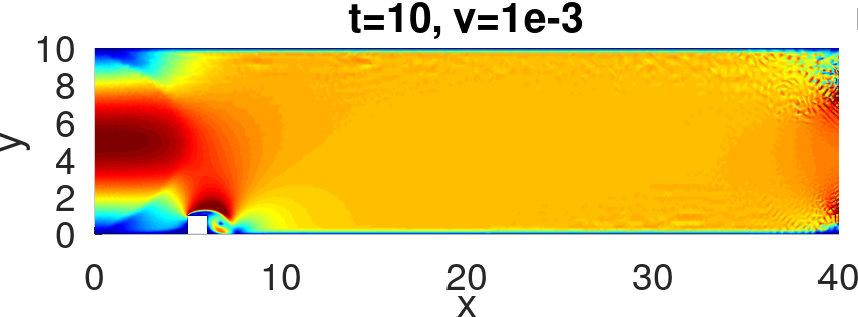}
		\includegraphics[width=0.44\textwidth]{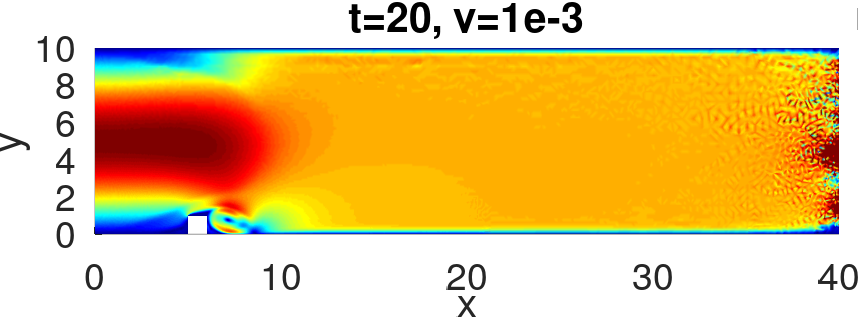}
		\includegraphics[width=0.1\textwidth]{colorbar.png}
	}\vspace{0ex}
	\centerline{
		\includegraphics[width=0.44\textwidth]{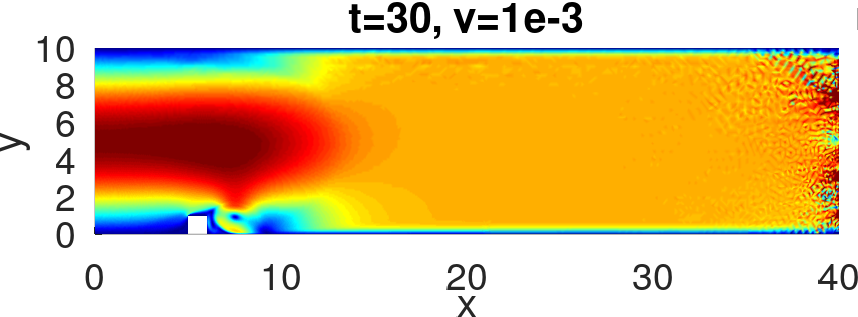}
		\includegraphics[width=0.44\textwidth]{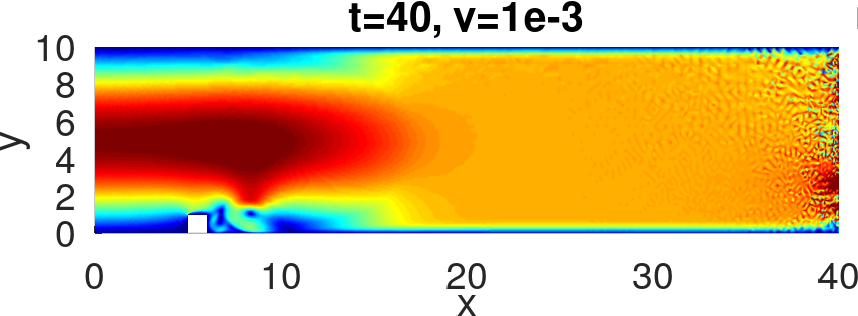}
		\includegraphics[width=0.1\textwidth]{colorbar.png}
	}
	\caption{Plots of magnitude of the simulated velocity at times $t$ = 10, 20, 30 and 40 generated by  LM-SKEW  for the channel flow past a forward-backward facing step problem with the viscosity $\Rey = 10^4$.}
	\label{step:mag:skew}
\end{figure}

\begin{figure}[!ht]
	\centerline{
		{\includegraphics[width=0.37\textwidth]{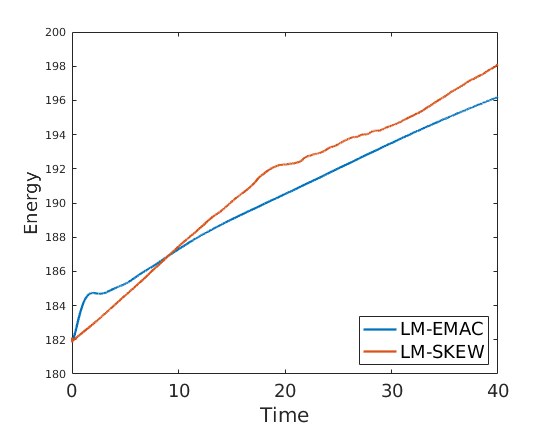}}\hspace{-2ex}
		{\includegraphics[width=0.37\textwidth]{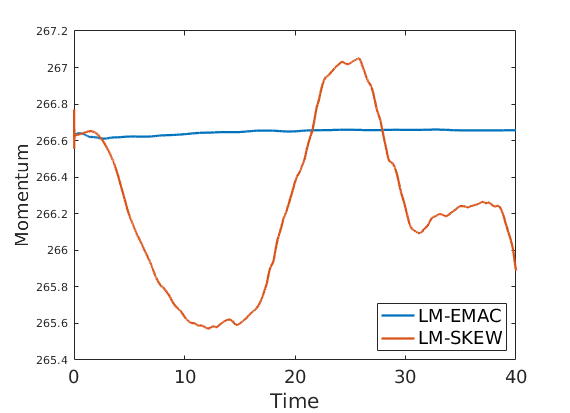}}\hspace{-2ex}
		{\includegraphics[width=0.37\textwidth]{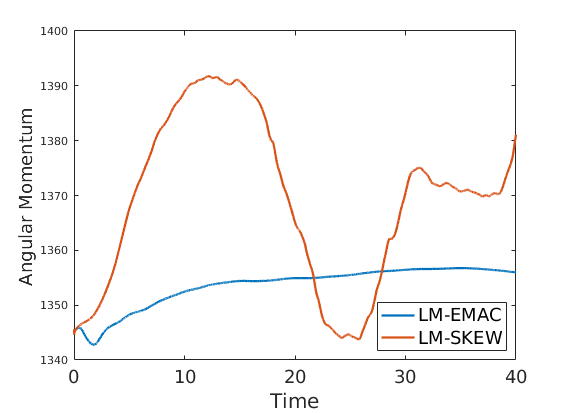}}
	}
	\vspace{-1ex}
	\caption{Evolutions of  the simulated kinetic energy (left), momentum (middle) and angular momentum (right) of the numerical solution generated by LM-EMAC and LM-SKEW schemes for the step problem with the viscosity $\Rey = 10^4$. The time step size $\delt=0.01$. }
	\label{step:quan:com}
\end{figure}

\section{Conclusions} \label{sec-Con}

We extended the long-time accuracy analysis of EMAC to the context of Large Eddy Simulations, focusing on the Ladyzhenskaya model, which can be viewed as a generalization of the widely used Smagorinsky model. The analysis revealed that the Gronwall constant in EMAC’s error bounds is significantly smaller than in SKEW, as EMAC’s bounds are not explicitly dependent on the Reynolds number. Numerical experiments confirmed the theoretical findings and aligned with results reported in the literature where EMAC consistently outperforms analogous methods based on SKEW particularly in long-time simulations. Future work will focus on developing sharper error estimates for the Ladyzhenskaya–EMAC system and exploring the impact of different stabilization parameters on long-time accuracy. A sensitivity study or parameter investigation for $ \delta^r$, and the power $ s$ could also merit further study.

\section*{Acknowledgments}
R. Lan's research is partially supported by Shandong Provincial Natural Science Fund for Excellent Young Scientists Fund Program (Overseas) under grant number 2023HWYQ-064, National Natural Science Foundation of China under grant number 12301531, the Shandong Provincial Youth Innovation Project under the grant number 2024KJN057 and  the OUC Scientific Research Program for Young Talented Professionals.

\section*{References}
\bibliographystyle{elsarticle-num}
\bibliography{ref}

\end{document}